\documentclass[11pt]{amsart}

\usepackage[T1]{fontenc}
\usepackage[utf8]{inputenc}
\usepackage{lmodern}
\usepackage{microtype}
\usepackage{amsmath,amssymb,amsthm,mathtools}
\usepackage{enumitem}
\usepackage{booktabs,array}
\usepackage{hyperref}
\usepackage{aliascnt}
\usepackage[ruled,vlined]{algorithm2e}
\usepackage[nameinlink,capitalize,noabbrev]{cleveref}
\usepackage[a4paper,margin=31mm]{geometry}
\usepackage{tikz}
\usepackage{parskip}
\usepackage{longtable}

\hypersetup{colorlinks=true,linkcolor=blue,citecolor=blue,urlcolor=blue}
\setlist{nosep}

\newtheorem{maintheorem}{Theorem}
\newaliascnt{maincorollary}{maintheorem}
\newtheorem{maincorollary}[maincorollary]{Corollary}
\aliascntresetthe{maincorollary}

\newtheorem{theorem}{Theorem}[section]

\newaliascnt{proposition}{theorem}
\newtheorem{proposition}[proposition]{Proposition}
\aliascntresetthe{proposition}

\newaliascnt{lemma}{theorem}
\newtheorem{lemma}[lemma]{Lemma}
\aliascntresetthe{lemma}

\newaliascnt{corollary}{theorem}

\aliascntresetthe{corollary}

\theoremstyle{definition}
\newaliascnt{definition}{theorem}
\newtheorem{definition}[definition]{Definition}
\aliascntresetthe{definition}

\newaliascnt{example}{theorem}
\newtheorem{example}[example]{Example}
\aliascntresetthe{example}

\theoremstyle{remark}
\newaliascnt{remark}{theorem}

\aliascntresetthe{remark}

\Crefname{maintheorem}{Theorem}{Theorems}
\Crefname{maincorollary}{Corollary}{Corollaries}
\Crefname{theorem}{Theorem}{Theorems}
\Crefname{proposition}{Proposition}{Propositions}
\Crefname{lemma}{Lemma}{Lemmas}
\Crefname{corollary}{Corollary}{Corollaries}
\Crefname{definition}{Definition}{Definitions}
\Crefname{example}{Example}{Examples}
\Crefname{remark}{Remark}{Remarks}
\Crefname{section}{Section}{Sections}
\Crefname{appendix}{Appendix}{Appendices}

\newcommand{\K}{\mathbb K}
\newcommand{\Z}{\mathbb Z}
\newcommand{\N}{\mathbb N}
\newcommand{\A}{\mathbb A}
\newcommand{\PP}{\mathbb P}
\newcommand{\Spec}{\operatorname{Spec}}
\newcommand{\Cox}{\mathcal{R}}
\newcommand{\Bl}{\operatorname{Bl}}
\newcommand{\gr}{\operatorname{gr}}
\newcommand{\Ass}{\operatorname{Ass}}
\newcommand{\htop}{\operatorname{ht}}
\newcommand{\rk}{\operatorname{rk}}

\newcommand{\Cl}{\operatorname{Cl}}
\newcommand{\cR}{\mathcal R}
\newcommand{\cE}{\mathcal E}

\newcommand{\m}{\mathfrak m}
\newcommand{\n}{\mathfrak n}
\newcommand{\p}{\mathfrak p}

\newcommand{\B}{\mathcal J_{\mathrm{irr}}}

\title[Multiplicity-one Cox rings of toric point blow-ups]
{Multiplicity-One Cox Rings of Toric Point Blow-Ups}
\author{Antonio Laface and Luca Ugaglia}
\date{}



\begin{document}

\begin{abstract}
Let $X$ be a complete toric variety and let $e$ be the identity of its
open torus.  The Cox ring of $\Bl_eX$ is described by saturated powers
of the toric-point lattice ideal $I_X$.  We give a finite criterion for
generation in Rees multiplicity one in terms of analytic spreads of
projected lattice ideals and zero divisors in the associated graded
ring.
We apply this criterion first to fake weighted projective spaces, proving
that multiplicity-one generation is equivalent to $I_X$ being a complete
intersection.  For projective toric surfaces, we prove that
multiplicity-one generation holds if and only if it holds for every
Picard-number-one toric surface dominated by $X$.
\end{abstract}

\maketitle

\section*{Introduction}
\label{sec:introduction}

Throughout the paper, $\K$ is an algebraically closed field of
characteristic zero.  Let $X=X_\Sigma$ be a positive-dimensional
complete toric variety.  Its Cox ring is the polynomial ring
\[
        R=\Cox(X)=\K[x_1,\ldots,x_r],
\]
whose variables correspond to the rays of $\Sigma$, and we denote its
irrelevant ideal by $\B(X)$.
The open torus of $X$ acts transitively on itself, so the blow-up at a
point in the open torus is independent, up to isomorphism, of the chosen point.  
We therefore blow up the identity $e$.  In Cox coordinates, the
characteristic-quasitorus orbit lying over $e$ is defined by a lattice
ideal $I_X\subseteq R$.  Its lattice $L_X\subseteq\Z^r$ is obtained by
pairing the characters of the open torus with the primitive ray
generators of $\Sigma$.  Since the fan is complete, $L_X$ is positive.
We call $I_X$ the \emph{toric-point ideal} of $X$.
The starting point of this paper is the description of the Cox ring of
the blow-up in terms of saturated powers of $I_X$.  By
Hausen--Keicher--Laface \cite{HKL-computing},
\[
 \Cox(\Bl_eX)
 \cong
 R\bigl[t,(I_X^m:\B(X)^\infty)t^{-m}:m\geq1\bigr].
\]
We say that this Cox ring is \emph{generated in Rees multiplicity one}
when the algebra on the right is generated by $R$, $t$, and
$I_Xt^{-1}$.  Equivalently, saturation introduces no new elements in
higher multiplicity:
\[
        I_X^m:\B(X)^\infty=I_X^m
        \qquad\text{for every }m\geq1.
\]
The difficulty is that this is an infinite family of conditions.

The problem has roots in the study of symbolic Rees algebras of space
monomial primes.  In the case of weighted projective planes, finite
generation of $\Cox(\Bl_eX)$ was investigated by Huneke
\cite{HunekeSymbolicBlowups,HunekeSymbolicPowers}, Srinivasan
\cite{Srinivasan}, Cutkosky \cite{Cutkosky}, and
Goto--Nishida--Shimoda \cite{GotoNishidaShimoda}.  Cutkosky related
these algebras to blow-ups of weighted projective planes at points of
the open torus, while Goto--Nishida--Watanabe
\cite{GotoNishidaWatanabe} constructed characteristic-zero examples
whose symbolic Rees algebras are not finitely generated.  The
connection with negative curves was subsequently developed by
Kurano--Matsuoka \cite{KuranoMatsuoka}, Kurano--Nishida
\cite{KuranoNishida}, Inagawa--Kurano \cite{InagawaKurano}, and others;
see also \cite{KuranoNegativeCurves}.

The language of Cox rings placed these questions in the broader
framework of birational geometry.  Castravet--Tevelev
\cite{CastravetTevelevMDS} used the Goto--Nishida--Watanabe examples
in their study of the Mori dream space property of
$\overline M_{0,n}$.  González--Karu
\cite{GonzalezKaru,GonzalezKaruHigher} developed toric constructions
of blow-ups with non-finitely generated Cox rings, while
Hausen--Keicher--Laface \cite{HKL-weighted,HKL-computing} introduced
computational methods based on saturated Rees algebras and studied
generators of low Rees multiplicity for weighted projective planes.
Further results on weighted projective spaces, toric surfaces,
negative curves, and extremal rays appear in
\cite{HeWeighted,HeToric,GonzalezAnayaFamily,GonzalezAnayaMDS,
GonzalezAnayaExtremal,CastravetLafaceTevelevUgaglia}.
For broader accounts of this circle of questions, see
\cite{CastravetSurvey,GonzalezLafaceSurvey}.

Most of this literature concerns finite generation of the Cox ring and
poliedrality of the effective cone.  We consider a stronger question:
when is the Cox ring generated in Rees multiplicity one?   The low-multiplicity
results of \cite{HKL-weighted} may be regarded as a Picard-rank-one
precursor of this problem.

Our criterion is formulated in terms of the coordinate boundary of
the total coordinate space $\overline X\simeq\mathbb K^r$.  For a subset
$A\subseteq[r]:=\{1,\ldots,r\}$, let $L_A\subseteq\Z^A$ be the
coordinate projection of $L_X$.  When $L_A$ is positive, the
corresponding lattice ideal $I_{L_A}$ describes the transverse
behavior of $I_X$ along the associated coordinate stratum.  The
relevant numerical invariant is its analytic spread
$\ell(I_{L_A})$ at the coordinate origin.  Its definition and its
geometric interpretation in terms of normal-cone fibers are given in
\Cref{sec:projected}.
On the global side, set
\[
        \gr_{I_X}(R) 
        :=
        \bigoplus_{m\geq0}I_X^m/I_X^{m+1},
        \qquad
        \Delta:=x_1\cdots x_r\in R,
\]
and let $\bar\Delta$ denote the class of $\Delta$ in
$\gr_{I_X}(R)$.  The scheme $\Spec \gr_{I_X}(R)$ is the normal cone of
$V(I_X)$ in $\Spec R$.
The following theorem is our main result.

\begin{maintheorem}
\label{thm:main-criterion}
With the above notation, the following conditions are equivalent:
\begin{enumerate}[label=\textup{(\roman*)}]
\item $\ell(I_{L_A})\leq |A|-1$ for every nonempty
      $A\subseteq[r]$ such that $L_A$ is positive;
\item $\dim(\gr_{I_X}(R)/\Delta \gr_{I_X}(R))<r$.
\end{enumerate}

If $\bar\Delta$ is not a zero divisor in $\gr_{I_X}(R)$, then
$\Cox(\Bl_eX)$ is generated in Rees multiplicity one.  In particular,
if $\gr_{I_X}(R)$ is unmixed, then the equivalent conditions
\textup{(i)}--\textup{(ii)} imply that $\bar\Delta$ is not a zero
divisor, and hence imply multiplicity-one generation.
\end{maintheorem}

The equivalence in \Cref{thm:main-criterion} has a direct geometric
meaning.  Let $O_A$ be the coordinate stratum on which precisely the
variables indexed by $A$ vanish.  This stratum meets $V(I_X)$ if and
only if $L_A$ is positive, and the part of the normal cone lying over
each irreducible component of the intersection has dimension
\[
        |A^c|+\ell(I_{L_A}).
\]
Thus the numerical conditions say that no top-dimensional component
of the normal cone is supported on the coordinate boundary $V(\Delta)$.  
They do not, by themselves, exclude associated primes of smaller dimension
supported there.  The additional requirement that $\bar\Delta$ not be
a zero divisor rules out all such associated primes.  When $\gr_{I_X}(R)$ is
unmixed, every associated prime has dimension $r$, so this additional
condition already follows from the numerical one.

Our first application concerns fake weighted projective spaces.  A
fake weighted projective space of dimension $d$ is a complete
$\mathbb Q$-factorial toric variety of Picard number one with $d+1$
rays.  Its class group may have torsion, but its Cox ring is a
polynomial ring in $d+1$ variables and its irrelevant ideal is the
homogeneous maximal ideal.  In this case the multiplicity-one problem
has an exact answer.

As a concrete application, we consider the $225$ canonical Fano
tetrahedra in dimension three.  For $131$ of them, Demazure-root
automorphisms reduce the blow-up to complexity at most one.  Among the
remaining $94$ cases, exactly $27$ have complete-intersection
toric-point ideal.  By our criterion, the other $67$ do not have Cox
ring generated in Rees multiplicity one.  Thus exactly $158$ of the
$225$ corresponding blow-ups have Cox ring generated in Rees
multiplicity one; see \Cref{prop:canonical-fano-tetrahedra}.

\begin{maintheorem}
\label{thm:main-fwps}
Let $X$ be a fake weighted projective space of dimension at least two,
and let $I_X$ be its toric-point ideal.  Then
$\Cox(\Bl_eX)$ is generated in Rees multiplicity one if and only if
$I_X$ is a complete intersection.
\end{maintheorem}

We next turn to projective toric surfaces.  For
$A\subseteq[r]$, put $P_A=(x_i:i\in A)$.  If $A$ consists of three
indices, then $L_A$ is positive precisely when the corresponding
three rays positively span $N_\mathbb R$.  They then form a complete
three-ray fan $\Sigma_A$, and the original fan refines $\Sigma_A$.
The identity on the lattice therefore induces a toric birational
contraction
\[
        \varphi_A:X\longrightarrow X_A:=X_{\Sigma_A},
\]
which we call a \emph{three-ray contraction}.

\begin{maintheorem}
\label{thm:main-surface}
Let $X$ be a projective toric surface with Cox ring
$\K[x_1,\ldots,x_r]$ and toric-point ideal $I_X$.  The following
conditions are equivalent:
\begin{enumerate}[label=\textup{(\roman*)}]
\item $\Cox(\Bl_eX)$ is generated in Rees multiplicity one;
\item $I_X\not\subseteq P_A^2$ for every three-element subset
      $A\subseteq[r]$ such that $L_A$ is positive;
\item $\Cox(\Bl_eX_A)$ is generated in Rees multiplicity one for every
      three-ray contraction $X\to X_A$.
\end{enumerate}
\end{maintheorem}

The implication \textup{(i)}$\Rightarrow$\textup{(ii)} was proved in
\cite[Theorem~3.5]{LU-minimal}.  The converse combines
\Cref{thm:main-criterion}, the analytic-spread theorem of
Hà--Morales for codimension-two lattice ideals
\cite{HaMorales}, and the complete-intersection criterion for
three-ray fans in \cite[Proposition~3.4]{LU-minimal}.  In this way,
the low-Rees-multiplicity problem studied for weighted projective
planes in \cite{HKL-weighted} becomes a local three-ray criterion for
arbitrary projective toric surfaces.

As an immediate consequence, we obtain the following result, which was conjectured in~\cite{LU-minimal}.

\begin{maincorollary}
\label{cor:main-minimal}
Let $X$ be a minimal complete toric surface of Picard rank two.  Then
$\Cox(\Bl_eX)$ is generated in Rees multiplicity one.
\end{maincorollary}

The surface criterion also complements the literature on finite
generation and negative curves for toric point blow-ups
\cite{GonzalezKaru,HeToric,GonzalezAnayaMDS,
GonzalezAnayaExtremal,CastravetLafaceTevelevUgaglia}.
As a further application, we study partial toric resolutions of
Gorenstein toric Fano surfaces.  We prove in
\Cref{prop:gorenstein-partial-resolutions} that multiplicity-one
generation fails exactly for those dominating a unique exceptional
rank-one surface.  For this surface,
\Cref{prop:exceptional-cox-ring} shows that the Cox ring of the blow-up
has one additional generator in Rees multiplicity two.

The paper is organized as follows.
\Cref{sec:normal-cone} develops the algebraic criteria involving
saturated Rees algebras, normal cones, and analytic spread.
\Cref{sec:projected} studies the normal cone over coordinate strata
and proves \Cref{thm:main-criterion}.
\Cref{sec:rank-one} treats fake weighted projective spaces.
Finally, \Cref{sec:surfaces} proves the surface theorem and its
applications.

\subsection*{Acknowledgements}
The authors were partially supported by Proyecto FONDECYT Regular
No.~1230287.
The second author is member of 
INdAM - GNSAGA

\section{Rees algebras, normal cones, and analytic spread}
\label{sec:normal-cone}

This section collects the commutative-algebra tools used throughout the
paper.  We first relate saturated extended Rees algebras to zero divisors
in the associated graded ring.  We then recall how analytic spread
describes the fibers of the normal cone, and conclude with a simple
splitting result for fiber cones that will be used in the local analysis
of coordinate strata.
For the standard facts on Rees algebras, associated graded rings, fiber
cones, and analytic spread used in this section, we refer to
\cite[Chapter~5]{HunekeSwanson} and \cite[Chapter~1]{Vasconcelos}.
For associated primes, zero divisors, and unmixed rings, see
\cite{BrunsHerzog}.  We use the usual interpretation of
$\Spec\gr_I(R)$ as the normal cone; see, for instance,
\cite[Chapter~5]{EisenbudCA}.

We begin with the algebraic setup.  Let $R$ be a Noetherian ring, and
let $I\subsetneq R$ and $\mathfrak b\subseteq R$ be ideals.  The ordinary
extended Rees algebra of $I$ and its $\mathfrak b$-saturated version are
\[
 \cE(I)=R[t,It^{-1}],
 \qquad
 \cE^{\mathrm{sat}}_{\mathfrak b}(I)
 =R\bigl[t,(I^m:\mathfrak b^\infty)t^{-m}:m\geq1\bigr]
 \subseteq R[t,t^{-1}].
\]
Thus $\cE(I)=\cE^{\mathrm{sat}}_{\mathfrak b}(I)$ precisely when every
power of $I$ is already $\mathfrak b$-saturated.  Since this involves
infinitely many powers, we pass to the associated graded ring
\[
        \gr_I(R)=\bigoplus_{m\geq0}I^m/I^{m+1},
\]
which records them simultaneously.  Its spectrum
$C_I:=\Spec\gr_I(R)$ is the normal cone of $V(I)$ in $\Spec R$, and for
every proper ideal $I$ one has the standard equality
$\dim\gr_I(R)=\dim R$.

\begin{proposition}
\label{prop:boundary-torsion}
Let $R$ be Noetherian, let $I\subsetneq R$ and
$\mathfrak b\subseteq R$ be ideals, and choose an element
$h\in\mathfrak b$. Let
\[
    \bar h\in R/I=\operatorname{gr}_I(R)_0
\]
denote the image of $h$. If $\bar h$ is a non-zero-divisor on
$\operatorname{gr}_I(R)$, then
\[
    \cE(I)=\cE^{\mathrm{sat}}_{\mathfrak b}(I).
\]
\end{proposition}

\begin{proof}
Set $A=\cE(I)$ and $B=\cE^{\mathrm{sat}}_{\mathfrak b}(I)$.  The first
observation is that $A_h=B_h$.  Indeed, if
$gt^{-m}\in B$, then $g\in I^m:\mathfrak b^\infty$.  Since
$h\in\mathfrak b$, a power of $h$ sends $g$ into $I^m$, and therefore
$gt^{-m}\in A_h$.
There is a natural isomorphism $A/tA\cong \operatorname{gr}_I(R)$. 
Under this identification, $\bar h$ is a non-zero-divisor on $\operatorname{gr}_I(R)$ 
if and only if $tA:h=tA$ (because multiplication by $h$ is injective), 
equivalently $tA:h^m=tA$ for every $m\geq 1$. 
Hence the hypothesis is equivalent to
\[
 tA:h^\infty=tA.
\]
It follows that $tB\cap A=tA$.  In fact, if $a\in tB\cap A$, then
the inclusion $B\subseteq B_h = A_h$ gives
\[
        a\in tA_h\cap A=tA:h^\infty=tA
        .
\]
The reverse inclusion is immediate.
Suppose now that $A\neq B$.  Choose the smallest integer $m\geq1$ for which
$B$ contains a homogeneous element $gt^{-m}$ not belonging to $A$.  The
nonnegative $t$-degree pieces of $A$ and $B$ are the same, so minimality gives
$gt^{-m+1}\in A$.  On the other hand,
$gt^{-m+1}=t(gt^{-m})\in tB$.  Thus $gt^{-m+1}\in tB\cap A=tA$, and we may
write $gt^{-m+1}=ta$ with $a\in A$.  Cancelling the invertible element $t$ in
$R[t,t^{-1}]$ yields $gt^{-m}=a\in A$, a contradiction.
\end{proof}

A Noetherian ring $G$ is called \emph{unmixed} if every prime in
$\operatorname{Ass}(G)$ has dimension $\dim G$.

\begin{lemma}
\label{lem:unmixed-check}
Let $G$ be an unmixed Noetherian ring of dimension $r$, and let $h\in G$.
If
\[
\dim(G/hG)<r,
\]
then $h$ is not a zero divisor in $G$.
\end{lemma}

\begin{proof}
An element of a Noetherian ring is a zero divisor if and only if it belongs
to an associated prime. Suppose that $h$ were a zero divisor in $G$. Then
$h\in\mathfrak q$ for some $\mathfrak q\in\operatorname{Ass}(G)$.
Since $G$ is unmixed, $\dim(G/\mathfrak q)=r$. As $h\in\mathfrak q$,
the quotient $G/\mathfrak q$ is a quotient of $G/hG$, and hence
$\dim(G/hG)\geq r$, contradicting the hypothesis.
\end{proof}

We recall the relation between analytic spread
and the fibers of the normal cone.
Let $(S,\mathfrak m)$ be a
Noetherian local ring and let $J\subseteq\mathfrak m$ be an ideal.  Its
{\em fiber cone} and {\em analytic spread} are defined by
\[
 \mathcal F_J(S)
 :=
 \bigoplus_{m\geq0}\frac{J^m}{\mathfrak mJ^m},
 \qquad
 \ell_S(J):=\dim\mathcal F_J(S).
\]
Equivalently,
$\mathcal F_J(S)\cong
\operatorname{gr}_J(S)\otimes_S S/\mathfrak m$.
In particular, for a prime $\mathfrak p\supseteq I$,
\[
 \ell_{R_\mathfrak p}(I_\mathfrak p)
 =
 \dim\!\left(
 \operatorname{gr}_{I_\mathfrak p}(R_\mathfrak p)
 \otimes_{R_\mathfrak p} k(\mathfrak p)
 \right)
 =
 \dim\!\left(
 \bigoplus_{m\geq0}
 \frac{I_\mathfrak p^m}
 {\mathfrak pR_\mathfrak p I_\mathfrak p^m}
 \right).
\]
This is the dimension of the fiber of the normal-cone
morphism $C_I\to V(I)$ over $\mathfrak p$.

We shall also need to understand how analytic spread changes when
independent regular parameters are added to an ideal.  The following
elementary lemma shows that, in this situation, the fiber cone simply
acquires one polynomial variable for each new parameter.

\begin{lemma}
\label{lem:fiber-cone-splitting}
Let $(S_0,\mathfrak m_0,k)$ be a Noetherian local ring and let
$J_0\subseteq\mathfrak m_0$ be an ideal.  Set
$S=S_0[[y_1,\ldots,y_c]]$, with maximal ideal
$\mathfrak m=\mathfrak m_0S+(y_1,\ldots,y_c)$, and let
$J=J_0S+(y_1,\ldots,y_c)$.  Then there is a natural isomorphism of
graded $k$-algebras
\[
        \mathcal F_J(S)
        \cong
        \mathcal F_{J_0}(S_0)[Y_1,\ldots,Y_c].
\]
In particular,
$\ell_S(J)=\ell_{S_0}(J_0)+c$.
\end{lemma}

\begin{proof}
We use the unique expansion of an element of $S$ as a power series
$\sum_{\alpha\in\mathbb N^c}f_\alpha y^\alpha$ with $f_\alpha\in S_0$.
Since $J=J_0S+(y_1,\ldots,y_c)$, we have
\[
        J^m=\sum_{|\alpha|\leq m}y^\alpha J_0^{m-|\alpha|}S.
\]
For $|\alpha|\leq m$, the coefficient of $y^\alpha$ in
$\mathfrak mJ^m$ is $\mathfrak m_0J_0^{m-|\alpha|}$.
Indeed, the contribution from $\mathfrak m_0SJ^m$ is
$\mathfrak m_0J_0^{m-|\alpha|}$. An element of
$(y_1,\ldots,y_c)J^m$ is a sum of terms $y_i g_i$ with $g_i\in J^m$;
to contribute to $y^\alpha$, the corresponding term of $g_i$ has
$y$-degree $|\alpha|-1$, so its coefficient lies in
$J_0^{m-|\alpha|+1} = J_0J_0^{m-|\alpha|}\subseteq 
\mathfrak m_0J_0^{m-|\alpha|}$.
Terms of $y$-degree greater than $m$ vanish in the quotient, since they lie in
$(y_1,\ldots,y_c)J^m$.  Therefore
\[
\frac{J^m}{\mathfrak mJ^m}
\cong
\bigoplus_{|\alpha|\leq m}
\frac{J_0^{m-|\alpha|}}
{\mathfrak m_0J_0^{m-|\alpha|}}Y^\alpha,
\]
where $Y_i$ denotes the class of $y_i$ in $J/\mathfrak mJ$.
Summing over all degrees $m$, we obtain
\[
 \mathcal F_J(S)
 =
 \bigoplus_{m\geq0}\frac{J^m}{\mathfrak mJ^m}
 \cong
 \bigoplus_{m\geq0}
 \bigoplus_{|\alpha|\leq m}
 \frac{J_0^{m-|\alpha|}}
      {\mathfrak m_0J_0^{m-|\alpha|}}\,Y^\alpha .
\]
Setting $n=m-|\alpha|$, this becomes
\[
 \bigoplus_{\alpha\in\mathbb N^c}
 \bigoplus_{n\geq0}
 \frac{J_0^n}{\mathfrak m_0J_0^n}\,Y^\alpha .
\]
The inner sum is $\mathcal F_{J_0}(S_0)$, while the monomials
$Y^\alpha$ are the monomials in $Y_1,\ldots,Y_c$.  These
identifications are compatible with multiplication, and therefore
\[
 \mathcal F_J(S)
 \cong
 \mathcal F_{J_0}(S_0)[Y_1,\ldots,Y_c].
\]
Taking Krull dimensions gives
$\ell_S(J)=\ell_{S_0}(J_0)+c$.
\end{proof}

\section{Coordinate strata and projected lattices}
\label{sec:projected}

We now study the normal cone over the coordinate strata of 
the total coordinate space.  Our goal is to translate the geometric condition in
\Cref{thm:main-criterion} into finitely many inequalities involving the
analytic spreads of projected lattice ideals.
We first recall a convenient presentation of the normal cone.  Let
$I=(f_1,\ldots,f_k)\subseteq R=\K[x_1,\ldots,x_r]$, set
$T=R[s_1,\ldots,s_k]$, and let $Q$ be the kernel of the map
$T\to R[u]$ sending $s_i$ to $f_i u$.  Its image is the Rees algebra
$\mathcal R(I)=R[Iu]$, so $T/Q\cong\mathcal R(I)$.
The associated graded ring is the quotient of the Rees algebra by the
ideal generated by $I$ in degree zero:
$\gr_I(R)\cong \mathcal R(I)/I\mathcal R(I)$.
Since the inverse image
of $I\mathcal R(I)$ in $T$ is $Q+IT$, it follows that
\begin{equation}
\label{eq:graded-rees-presentation}
        \gr_I(R)\cong T/(Q+IT).
\end{equation}
Thus the normal cone $C_I=\Spec\gr_I(R)$ is realized as a closed
subscheme of $\Spec T$.
We next decompose the base into coordinate strata.  For
$A\subseteq[r]$, let $P_A=(x_i:i\in A)$ and write
$x_{A^c}=\prod_{j\notin A}x_j$.  The stratum
\[
        O_A=V(P_A)\cap D(x_{A^c})
\]
is the locus where exactly the coordinates indexed by $A$ vanish.  We
will describe the part of the normal cone lying over
$V(I)\cap O_A$ and then specialize this description to lattice ideals.

\begin{proposition}
\label{prop:coordinate-stratum}
Under the presentation \eqref{eq:graded-rees-presentation}, the closure
in $\Spec T$ of the part of the normal cone lying over
$V(I)\cap O_A$ is defined by
\[
        D_A=(Q+IT+P_AT):x_{A^c}^\infty.
\]
Consequently, if $\mathfrak a\subsetneq R$ is a monomial ideal, then
\begin{equation}
\label{eq:monomial-strata-general}
 \dim\bigl(\gr_I(R)/\mathfrak a\gr_I(R)\bigr)
 =
 \max_{\substack{\mathfrak a\subseteq P_A\\D_A\neq T}}
 \dim(T/D_A).
\end{equation}
For $\mathfrak a=(\Delta)$, where $\Delta=x_1\cdots x_r$, the maximum
is taken over all nonempty subsets $A\subseteq[r]$.
\end{proposition}

\begin{proof}
Adding $P_AT$ to $Q+IT$ restricts the base of the normal cone to the
coordinate subspace $V(P_A)$, while localizing at $x_{A^c}$ restricts to
$O_A$.  Contracting the resulting ideal from $T_{x_{A^c}}$ back to $T$
amounts to saturation by $x_{A^c}$, which gives $D_A$.

Every point of $\Spec R$ belongs to a unique coordinate stratum.  Since
$\mathfrak a$ is monomial, $O_A$ is contained in $V(\mathfrak a)$
exactly when $\mathfrak a\subseteq P_A$.  The inverse image of
$V(\mathfrak a)$ in the normal cone is therefore the finite union of
the corresponding locally closed pieces, and its dimension is the
maximum of the dimensions of their closures.  This proves
\eqref{eq:monomial-strata-general}.  Finally,
$(\Delta)\subseteq P_A$ if and only if $A\neq\varnothing$.
\end{proof}

We now specialize to lattice ideals.  Let $L\subseteq\Z^r$ be a
lattice and set
\[
        I_L=(x^{u^+}-x^{u^-}:u\in L)\subseteq R.
\]
We shall use the standard facts that
$\operatorname{ht}(I_L)=\operatorname{rk}(L)$ and that, over our
ground field, lattice ideals are radical; see
\cite[Section~2]{EisenbudSturmfels}.  We call $L$ \emph{positive} if
$L\cap\N^r=\{0\}$.  Equivalently, there is a vector in
$\mathbb Q_{>0}^r$ orthogonal to $L$, so $I_L$ is homogeneous for a
positive grading.

For a coordinate set $A\subseteq[r]$, we write
\[
        L_A=\pi_A(L)\subseteq\Z^A,
        \qquad
        L^{A^c}=L\cap\Z^{A^c},
\]
for the coordinate projection of $L$ and the sublattice supported on
$A^c$, respectively.  
The lattice ideal has a natural diagonal symmetry.  Let $\bar e_i$
denote the class of the $i$-th standard basis vector in $\Z^r/L$ and
set
\[
        G_L=\operatorname{Hom}(\Z^r/L,\K^*).
\]
The group $G_L$ acts on $\A^r=\Spec R$ by
$g\cdot(a_1,\ldots,a_r)
=(g(\bar e_1)a_1,\ldots,g(\bar e_r)a_r)$.
Equivalently, this is the action associated with the
$\Z^r/L$-grading of $R$ given by $\deg(x_i)=\bar e_i$.
Let
\[
        \pi:C_{I_L}=\Spec\gr_{I_L}(R)\longrightarrow V(I_L)
\]
be the normal-cone morphism.  For each $A\subseteq[r]$, we denote by
$Z_A=V(I_L)\cap O_A$ the part of the Cox orbit lying in the
corresponding coordinate stratum, and set
\[
        H_A:=\operatorname{Hom}(\Z^{A^c}/L^{A^c},\K^*).
\]
The following proposition shows, in particular, that whenever $Z_A$
is nonempty it is homogeneous under $G_L$; consequently, it is enough
to compute the normal-cone fiber at a single point of the stratum.

\begin{proposition}
\label{prop:quasitorus-strata}
Assume that $L$ is positive.  Then:
\begin{enumerate}[label=\textup{(\roman*)}]
\item the $G_L$-action preserves $V(I_L)$ and $O_A$, and lifts to
      $C_{I_L}$ so that $\pi$ is equivariant;

\item $Z_A\neq\varnothing$ if and only if $L_A$ is positive;

\item if $L_A$ is positive, then $Z_A\cong H_A$ and $G_L$ acts
      transitively on $Z_A$; in particular, $Z_A$ is pure of dimension
      \[
       |A^c|-\rk(L^{A^c})
      \]
      and the fibers of $\pi$ over points of $Z_A$ are mutually isomorphic.
\end{enumerate}
\end{proposition}

\begin{proof}
\textup{(i)}
For $g\in G_L$, let $\alpha_g$ be the automorphism of $R$ defined by
$\alpha_g(x_i)=g(\bar e_i)x_i$.  If $u\in L$, then $u^+$ and $u^-$
have the same class in $\Z^r/L$, and hence
\[
 \alpha_g(x^{u^+}-x^{u^-})
 =
 g(\overline{u^+})(x^{u^+}-x^{u^-}).
\]
Thus $I_L$, and therefore every power of $I_L$, is $G_L$-stable.
The action consequently descends to $\gr_{I_L}(R)$.  On its
degree-zero part it agrees with the action on $R/I_L$, so the morphism
induced by
$R/I_L=(\gr_{I_L}(R))_0\hookrightarrow\gr_{I_L}(R)$ is
$G_L$-equivariant.  Finally, $O_A$ is preserved because the action
only rescales coordinates by nonzero constants.

\textup{(ii)}
Suppose first that $L_A$ is not positive.  Choose
$0\neq a\in L_A\cap\N^A$ and lift it to $u\in L$.  After possibly
replacing $u$ by $-u$, one monomial of the corresponding lattice
binomial contains an $A$-variable, whereas the other contains none.
On $O_A$ the former vanishes and the latter is a nonzero Laurent
monomial in the $A^c$-variables.  Hence $Z_A$ is empty.

Conversely, assume that $L_A$ is positive.  If the $A$-part of
$u\in L$ is nonzero, then it has both positive and negative
coordinates; otherwise it or its negative would give a nonzero element
of $L_A\cap\N^A$.  Thus both monomials of the associated binomial
vanish on $O_A$.  The only equations that survive on $O_A$ are
therefore those corresponding to $u\in L^{A^c}$.  After identifying
$O_A$ with the torus $(\K^*)^{A^c}$, these equations are
$x^u=1$ for $u\in L^{A^c}$, so
\[
        Z_A\cong
        \operatorname{Hom}(\Z^{A^c}/L^{A^c},\K^*)
        =H_A.
\]
In particular, $Z_A$ is nonempty.

\textup{(iii)}
The inclusion $\Z^{A^c}\hookrightarrow\Z^r$ induces an injection
$\Z^{A^c}/L^{A^c}\hookrightarrow\Z^r/L$.
Since $\K^*$ is divisible, every character of
$\Z^{A^c}/L^{A^c}$ extends to $\Z^r/L$.  Restriction of characters
therefore gives a surjective homomorphism
$G_L\to H_A$.
Under the identification $Z_A\cong H_A$ obtained in \textup{(ii)},
this is precisely the action of $G_L$ on $Z_A$.  Hence the action is
transitive.  Since $H_A$ is a diagonalizable group with character
group $\Z^{A^c}/L^{A^c}$, all its irreducible components have dimension
\[
        \rk(\Z^{A^c}/L^{A^c})
        =
        |A^c|-\rk(L^{A^c}).
\]
\end{proof}

It remains to determine the common dimension of these fibers.  The
local equations transverse to the torus orbit may acquire nontrivial
coefficients; these are conveniently recorded by a partial character.

\begin{definition}
Let $F$ be a field, let $H\subseteq\Z^A$ be a lattice, and let
$\rho:H\to F^*$ be a homomorphism.  The corresponding
\emph{partial-character lattice ideal} is
\[
        I_{H,\rho}
        =
        \bigl(x^{u^+}-\rho(u)x^{u^-}:u\in H\bigr)
        \subseteq F[x_i:i\in A].
\]
The ordinary lattice ideal $I_H$ corresponds to the trivial character.
\end{definition}

\begin{lemma}
\label{lem:remove-character}
Let $F$ be a field.
Let $H\subseteq\Z^A$ be positive and let $\rho:H\to F^*$ be a partial
character.  After a finite field extension $F\subseteq F'$, a diagonal
change of variables identifies
$I_{H,\rho}F'[x_i:i\in A]$ with $I_HF'[x_i:i\in A]$.  In particular,
the two ideals have the same analytic spread at the coordinate maximal
ideal.
\end{lemma}

\begin{proof}
Choose a Smith normal form for the inclusion $H\subseteq\Z^A$.
After adjoining finitely many roots, $\rho$ extends to a character of
$\Z^A$.  If this extension is written as
$u\mapsto\prod_{i\in A}c_i^{u_i}$, the diagonal rescaling
$x_i\mapsto c_i x_i$ removes the coefficients $\rho(u)$ from all the
lattice binomials.  This rescaling preserves the coordinate maximal
ideal, while finite extension of the residue field does not change the
dimension of the fiber cone.
\end{proof}

We can now compute the normal-cone fiber along a coordinate stratum.
The key point is that, after completion, the equations coming from the
torus orbit separate from those determined by the projected lattice.

For $A\subseteq[r]$, set
$S_A=\K[x_i:i\in A]$ and
$\mathfrak m_A=(x_i:i\in A)$.  Whenever $L_A$ is positive, we write
\[
 \ell(I_{L_A})
 :=
 \ell_{(S_A)_{\mathfrak m_A}}
 \bigl((I_{L_A})_{\mathfrak m_A}\bigr)
\]
for the analytic spread of $I_{L_A}$ at the coordinate maximal ideal.

\begin{proposition}
\label{prop:local-product}
Assume that $L$ and $L_A$ are positive. Let $Z$ be an irreducible
component of $V(I_L)\cap O_A$, and let $\mathfrak p_Z\subset R$ be the
prime ideal defining the closure of $Z$ in $\Spec R$. Then
\begin{equation}
\label{eq:local-spread-splitting}
 \ell_{R_{\mathfrak p_Z}}\bigl((I_L)_{\mathfrak p_Z}\bigr)
 =
 \rk(L^{A^c})+\ell(I_{L_A}).
\end{equation}
\end{proposition}

\begin{proof}
Set $c=\rk(L^{A^c})$ and let $\kappa=\kappa(Z)$ be the residue field of
$R_{\mathfrak p_Z}$. On the stratum $O_A$, the variables $x_i$ with
$i\in A$ vanish, whereas those with $i\in A^c$ are invertible.

For $u\in L$, write $a=\pi_A(u)$ and $b=\pi_{A^c}(u)$. Since $a$ and
$b$ have disjoint supports, the lattice binomial associated with $u$ is
\[
x^{u^+}-x^{u^-}
=
x_A^{a^+}x_{A^c}^{b^+}
-
x_A^{a^-}x_{A^c}^{b^-}.
\]
If $a\neq 0$, then $a\in L_A\setminus\{0\}$. Since $L_A$ is positive,
$a$ has both a positive and a negative component, so
$a^+\neq 0$ and $a^-\neq 0$. Consequently, each of the two monomials
above contains a variable indexed by $A$, and hence both vanish on
$O_A$.

If $a=0$, then $u\in L^{A^c}$ and the binomial involves only the
$A^c$-variables. In this case it factors as
\[
x_{A^c}^{b^+}-x_{A^c}^{b^-}
=
x_{A^c}^{b^-}\bigl(x_{A^c}^{b}-1\bigr).
\]
All variables indexed by $A^c$ are nonzero on $O_A$, so
$x_{A^c}^{b^-}$ is invertible in the coordinate ring of $O_A$.
Therefore, on $O_A$, this binomial vanishes if and only if
$x_{A^c}^{b}=1$, equivalently $x^u=1$. Thus, inside the torus
$(\K^*)^{A^c}$, the intersection $V(I_L)\cap O_A$ is defined precisely
by the equations
\[
x^u=1,
\qquad
u\in L^{A^c}.
\]
Choose a basis $u_1,\ldots,u_c$ of $L^{A^c}$ and set
$\chi_j=x^{u_j}$ and $y_j=\chi_j-1$. These $c$ equations generate the
restriction of $I_L$ to the $A^c$-torus. Moreover, their logarithmic
differentials $d\chi_j/\chi_j$ correspond to the linearly independent
vectors $u_j$. Since $Z$ has codimension $c$ in this torus by
\Cref{prop:quasitorus-strata}, the elements $y_1,\ldots,y_c$ are local
parameters transverse to $Z$.

Each $y_j$ belongs to $(I_L)_{\mathfrak p_Z}$. Indeed, the lattice
binomial associated with $u_j$ is
\[
x^{u_j^-}(\chi_j-1)=x^{u_j^-}y_j,
\]
and $x^{u_j^-}$ is a unit in $R_{\mathfrak p_Z}$ because it involves
only variables indexed by $A^c$.

The variables $x_i$, $i\in A$, are parameters transverse to the
coordinate stratum, while $y_1,\ldots,y_c$ are parameters transverse
to $Z$ inside that stratum. They therefore form a regular system of
parameters of the regular local ring $R_{\mathfrak p_Z}$. 
Since $R_{\mathfrak p_Z}$ is a regular local ring with residue field
$\kappa$, its completion is a complete equicharacteristic regular local
ring. By the Cohen structure theorem, it contains a coefficient field
isomorphic to $\kappa$, and, since the elements
$x_i$, $i\in A$, together with $y_1,\ldots,y_c$ form a regular system
of parameters, we obtain
\[
\widehat{R_{\mathfrak p_Z}}
\cong
\kappa[[x_i\ (i\in A),y_1,\ldots,y_c]].
\]
In particular, quotienting by $(y_1,\ldots,y_c)$ gives
$\kappa[[x_i:i\in A]]$.
We now determine the image of $(I_L)_{\mathfrak p_Z}$ in this
quotient. Let $u\in L$, with $a=\pi_A(u)$ and
$b=\pi_{A^c}(u)$ as above. Since every $A^c$-variable is a unit, the
corresponding lattice binomial is, up to multiplication by a unit,
\[
x_A^{a^+}-x_{A^c}^{-b}x_A^{a^-}.
\]
Modulo $(y_1,\ldots,y_c)$, the unit $x_{A^c}^{-b}$ becomes its residue
in $\kappa^*$. Denote this residue by $\rho_Z(a)$. It depends only on
$a$: if two elements of $L$ have the same projection to
$\mathbb{Z}^A$, their difference belongs to $L^{A^c}$, and the
corresponding character is equal to $1$ on $Z$. The multiplicativity
of characters also shows that $\rho_Z:L_A\to\kappa^*$ is a
homomorphism.

Thus every lattice binomial maps to a generator
$x_A^{a^+}-\rho_Z(a)x_A^{a^-}$ of the partial-character ideal
$I_{L_A,\rho_Z}$. Conversely, every $a\in L_A$ has a lift in $L$, so
all generators of $I_{L_A,\rho_Z}$ occur in this way. Since
$(y_1,\ldots,y_c)$ is already contained in the localized ideal, this
proves that
\[
(I_L)_{\mathfrak p_Z}\widehat{R_{\mathfrak p_Z}}
=
(y_1,\ldots,y_c)
+
I_{L_A,\rho_Z}\widehat{R_{\mathfrak p_Z}}.
\]

Let $\kappa'/\kappa$ be the finite extension provided by
\Cref{lem:remove-character}. After extending the coefficient field
from $\kappa$ to $\kappa'$ and applying the corresponding diagonal
change of the variables $x_i$, $i\in A$, the partial-character ideal
$I_{L_A,\rho_Z}$ becomes the ordinary lattice ideal $I_{L_A}$. The
extended ideal in
$\kappa'[[x_i\ (i\in A),y_1,\ldots,y_c]]$ is therefore
\[
(y_1,\ldots,y_c)
+
I_{L_A}\cdot
\kappa'[[x_i\ (i\in A),y_1,\ldots,y_c]].
\]

None of the operations just performed changes analytic spread.
Passing from a local ring to its completion leaves its fiber cone
unchanged. Extending the residue field from $\kappa$ to $\kappa'$
only tensors the fiber cone with $\kappa'$, and hence preserves its
Krull dimension. Finally, the diagonal change of variables is a local
automorphism and therefore identifies the corresponding fiber cones.

We may consequently compute the analytic spread after these
operations. By \Cref{lem:fiber-cone-splitting}, adjoining the
independent parameters $y_1,\ldots,y_c$ contributes exactly $c$ to the
analytic spread. The remaining term is the analytic spread of
$I_{L_A}$ in $\kappa'[[x_i:i\in A]]$. This is equal to
$\ell(I_{L_A})$, because extending the coefficient field from $\K$ to
$\kappa'$ base-changes the fiber cone of $I_{L_A}$, while passing to
the completed local ring does not change it. Hence
\[
\ell_{R_{\mathfrak p_Z}}\bigl((I_L)_{\mathfrak p_Z}\bigr)
=
c+\ell(I_{L_A})
=
\rk(L^{A^c})+\ell(I_{L_A}),
\]
as claimed.
\end{proof}

\begin{example}
\label{ex:local-product}
We illustrate \Cref{prop:local-product} on a toric threefold. Let
$N=\Z^3$ and consider
\[
V=
\left(
\begin{array}{rrrrr}
 -2 & 0 &  2 & 0 &  1\\
 -1 & 1 & -1 & 0 &  0\\
  0 & 0 &  0 & 1 & -1
\end{array}
\right).
\]
Its columns $v_1,\ldots,v_5$ are primitive and generate the rays of a
complete three-dimensional fan: the first three form a complete fan in
the plane $z=0$, and the maximal cones are obtained by adjoining
$v_4$ or $v_5$ to each pair of consecutive rays among
$v_1,v_2,v_3$. Let $X$ be the corresponding toric variety and write
$L=L_X\subseteq\Z^5$ for its orbit lattice. The rows of $V$ generate
$L$, and hence
\[
I_X=I_L=
(x_3^2x_5-x_1^2,\;x_2-x_1x_3,\;x_4-x_5).
\]
Take $A=\{1,2,3\}$ and $A^c=\{4,5\}$. Then
$L^{A^c}=\Z(0,0,0,1,-1)$ has rank one, while
$L_A=\langle(-2,0,2),(-1,1,-1)\rangle\subseteq\Z^3$ is a positive
lattice of rank two. It is not saturated, since the gcd of the maximal
minors of the two generators is $2$.

On the stratum $O_A$, the variables $x_1,x_2,x_3$ vanish and
$x_4,x_5$ are invertible. Thus the first two generators of $I_X$
vanish identically, while the third gives $x_4=x_5$. Hence
$Z=V(I_X)\cap O_A$ is a one-dimensional torus, and the prime ideal
defining its closure is
$\mathfrak p_Z=(x_1,x_2,x_3,x_4-x_5)$. Its residue field is
$\kappa(Z)=\K(t)$, where $t$ denotes the common residue of $x_4$ and
$x_5$.

We now follow the proof of \Cref{prop:local-product}. A basis of
$L^{A^c}$ is given by $u=(0,0,0,1,-1)$, whose character on the
$A^c$-torus is $\chi=x_4x_5^{-1}$. Since $\chi=1$ on $Z$, we take
$y=\chi-1$ as the parameter transverse to $Z$. We have
$x_4-x_5=x_5y$, and $x_5$ is a unit in $R_{\mathfrak p_Z}$.
Therefore $x_4-x_5$ and $y$ generate the same ideal locally, and in
particular $y\in(I_X)_{\mathfrak p_Z}$.

The elements $x_1,x_2,x_3,y$ form a regular system of parameters of
$R_{\mathfrak p_Z}$. By the Cohen structure theorem,
\[
\widehat{R_{\mathfrak p_Z}}
\cong
\K(t)[[x_1,x_2,x_3,y]].
\]
Since $x_5=t$ and $x_4=t(1+y)$ in the completion, the extended ideal is
\[
 I_X\widehat{R_{\mathfrak p_Z}}
=(y,\;t x_3^2-x_1^2,\;x_2-x_1x_3).
\]
The generator $y$ is the contribution of $L^{A^c}$. Modulo $(y)$, the
remaining two generators are, up to multiplication by units,
$x_3^2-t^{-1}x_1^2$ and $x_2-x_1x_3$. They therefore define the
partial-character lattice ideal associated with $L_A$, where
$\rho_Z(-2,0,2)=t^{-1}$ and
$\rho_Z(-1,1,-1)=1$.

This example also shows why a finite extension of the residue field
may be needed. Since $L_A$ has index two in its saturation, removing
the coefficient $t^{-1}$ requires, in general, a square root of $t$.
After adjoining $s$ with $s^2=t$ and making the diagonal change of
variables $x_3'=s x_3$ and $x_2'=s x_2$, the last two generators become,
up to multiplication by units, $x_3'^2-x_1^2$ and
$x_2'-x_1x_3'$. Thus the localized ideal acquires the product form
$(y)+I_{L_A}$.

By \Cref{lem:fiber-cone-splitting}, the parameter $y$ contributes one
to the analytic spread. Moreover,
$I_{L_A}=(x_3^2-x_1^2,\;x_2-x_1x_3)$ is a height-two complete
intersection, so $\ell(I_{L_A})=2$. Therefore
\[
\ell_{R_{\mathfrak p_Z}}\bigl((I_X)_{\mathfrak p_Z}\bigr)
=
\rk(L^{A^c})+\ell(I_{L_A})
=
1+2
=
3.
\]

Finally, $\dim Z=|A^c|-\rk(L^{A^c})=1$, so the corresponding part of
the normal cone has dimension $1+3=4$, equivalently
$|A^c|+\ell(I_{L_A})=2+2=4$.
\end{example}

We can now read off the dimension of the normal cone over every
coordinate stratum.

\begin{theorem}
\label{thm:projected-lattice-formula}
Let $L\subseteq\Z^r$ be positive.  If $L_A$ is positive, then
\begin{equation}
\label{eq:stratum-normal-cone-dimension}
 \dim\pi^{-1}\bigl(V(I_L)\cap O_A\bigr)
 =
 |A^c|+\ell(I_{L_A}).
\end{equation}
Consequently, for every proper monomial ideal
$\mathfrak a\subsetneq R$,
\begin{equation}
\label{eq:monomial-boundary-dimension}
 \dim\bigl(\gr_{I_L}(R)/\mathfrak a\gr_{I_L}(R)\bigr)
 =
 \max_{\substack{\mathfrak a\subseteq P_A\\
                  L_A\text{ positive}}}
 \bigl\{|A^c|+\ell(I_{L_A})\bigr\}.
\end{equation}
In particular, the left-hand side is smaller than $r$ if and only if
$\ell(I_{L_A})\leq |A|-1$ for every $A$ occurring in the maximum.
\end{theorem}

\begin{proof}
Let $Z$ be an irreducible component of $V(I_L)\cap O_A$.  By
\Cref{prop:quasitorus-strata}, its dimension is
$|A^c|-\rk(L^{A^c})$, and the normal-cone fibers have constant
dimension along the entire stratum.  Their dimension may therefore be
computed at $\mathfrak p_Z$.  The fiber there is the special fiber of
$(I_L)_{\mathfrak p_Z}$, so \Cref{prop:local-product} gives its dimension as
$\rk(L^{A^c})+\ell(I_{L_A})$.
Adding the dimension of $Z$ and the common fiber dimension gives
\eqref{eq:stratum-normal-cone-dimension}.  Formula
\eqref{eq:monomial-boundary-dimension} follows from
\Cref{prop:coordinate-stratum}.  Finally,
$|A^c|=r-|A|$, so
$|A^c|+\ell(I_{L_A})<r$ is equivalent to
$\ell(I_{L_A})\leq|A|-1$.
\end{proof}

\begin{proof}[Proof of \Cref{thm:main-criterion}]
Apply \Cref{thm:projected-lattice-formula} to $L=L_X$ and
$\mathfrak a=(\Delta)$.  By \Cref{prop:coordinate-stratum}, the
relevant coordinate sets are precisely the nonempty subsets of
$[r]$.  Formula \eqref{eq:monomial-boundary-dimension} therefore proves
the equivalence of \textup{(i)} and \textup{(ii)}.

Let $\bar\Delta$ denote the image of $\Delta=x_1\cdots x_r$ in
$R/I_X=(\gr_{I_X}(R))_0$.  Notice first that
$\Delta\in\B(X)$.  Indeed, the irrelevant ideal $\B(X)$ is generated
by square-free monomials, each of which divides $\Delta$.  If
$\bar\Delta$ is not a zero divisor in $\gr_{I_X}(R)$, then
\Cref{prop:boundary-torsion}, applied with $I=I_X$,
$\mathfrak b=\B(X)$, and $h=\Delta$, yields
\[
\cE(I_X)=\cE^{\mathrm{sat}}_{\B(X)}(I_X).
\]
By the description of the Cox ring of $\Bl_eX$ via the
$\B(X)$-saturated extended Rees algebra, this equality is precisely the
statement that $\Cox(\Bl_eX)$ is generated in multiplicity one.

Finally, suppose that $\gr_{I_X}(R)$ is unmixed and that the equivalent
numerical conditions hold.  Since $\dim \gr_{I_X}(R)=\dim R=r$, condition
\textup{(ii)} gives $\dim(\gr_{I_X}(R)/\Delta \gr_{I_X}(R))<r$.
By \Cref{lem:unmixed-check}, $\bar\Delta$ is not a zero divisor in
$\gr_{I_X}(R)$, and the preceding paragraph completes the proof.
\end{proof}

\section{Fake weighted projective spaces}
\label{sec:rank-one}

Let $X$ be a fake weighted projective space of dimension $d\geq2$.  Its Cox
ring and irrelevant ideal have the form
\[
        R=\K[x_0,\ldots,x_d],
        \qquad
        \B(X)=\m=(x_0,\ldots,x_d).
\]
Although $\Cl(X)$ may have torsion, the free part of the Cox grading gives a
positive $\Z$-grading on $R$.  The toric-point ideal $I_X$ is homogeneous for
this grading, has height $d$, and therefore defines a one-dimensional scheme
in $\A^{d+1}$.

We use the following consequence of the theorem of Cowsik--Nori.  In the form
needed here, it says that if $(A,\n)$ is a regular local ring with infinite
residue field, $J\subset A$ is radical, and $A/J^m$ is Cohen--Macaulay for
all sufficiently large $m$, then $J$ is a complete intersection; see
\cite[Corollary, p.~219]{CowsikNori}.  The generic complete-intersection
hypothesis in the cited formulation is automatic for a radical ideal in a
regular local ring after localization at a minimal prime.

\begin{proof}[Proof of \Cref{thm:main-fwps}]
Since $\B(X)=\m$, multiplicity-one
generation is equivalent to
\begin{equation}\label{eq:rank-one-saturation}
        I_X^m:\m^\infty=I_X^m
        \qquad\text{for every }m\geq1.
\end{equation}

Assume first that \eqref{eq:rank-one-saturation} holds.  Lattice ideals over
an algebraically closed field of characteristic zero are radical
\cite[Section~2]{EisenbudSturmfels}; hence $R/I_X$ is a one-dimensional
positively graded reduced ring.  A homogeneous prime containing $I_X$ is
either minimal over $I_X$ or equal to the homogeneous maximal ideal $\m$.
Indeed, any nonminimal homogeneous prime gives a zero-dimensional positively
graded domain, and such a quotient is concentrated in degree zero.
For every $m$, the saturation condition excludes $\m$ from
$\Ass_R(R/I_X^m)$.  After localization at $\m$, the quotient
\[
        R_\m/(I_XR_\m)^m
\]
has dimension one and its maximal ideal is not associated.  It is therefore
Cohen--Macaulay.  The ideal $I_XR_\m$ is radical.  At each of its minimal
primes $\p R_\m$, localization gives the maximal ideal of the regular local
ring $R_\p$, so it is generically a complete intersection.  The
Cowsik--Nori theorem now implies that $I_XR_\m$ is a complete intersection.
It remains to return from the local ring to the graded polynomial ring.  By
graded Nakayama,
\[
 \mu_R(I_X)=\dim_\K I_X/\m I_X
 =\mu_{R_\m}(I_XR_\m).
\]
Thus $I_X$ is generated by $\htop(I_X)=d$ homogeneous elements.  Since $R$
is Cohen--Macaulay, an ideal whose height equals its number of generators is
generated by a regular sequence; see \cite[Section~2.1]{BrunsHerzog}.
Hence $I_X$ is a complete intersection.

Conversely, suppose that $I_X$ is a complete intersection, generated by a
regular sequence $f_1,\ldots,f_d$.  For a complete-intersection ideal, the canonical map
\[
\operatorname{Sym}_{R/I_X}(I_X/I_X^2)
\longrightarrow
\gr_{I_X}(R)
\]
is an isomorphism. Since $f_1,\ldots,f_d$ is a regular sequence, their
classes form a basis of the free $R/I_X$-module $I_X/I_X^2$.  Hence
$\operatorname{Sym}_{R/I_X}(I_X/I_X^2)$ is a polynomial ring in $d$ variables over
$R/I_X$, and therefore
\[
        \gr_{I_X}(R)\cong (R/I_X)[T_1,\ldots,T_d].
\]
Since $R/I_X$ is one-dimensional and reduced, it is Cohen--Macaulay;
hence $\gr_{I_X}(R)$ is Cohen--Macaulay, and in particular unmixed.
In the present rank-one situation, the only subset $A\subseteq[r]$
for which $L_A$ is positive is $A=[r]$.  Moreover,
$\ell(I_X)=\htop(I_X)=d=r-1$.  Thus condition~\textup{(i)} of
\Cref{thm:main-criterion} holds.  The theorem then shows that
$\bar\Delta$ is not a zero divisor in $\gr_{I_X}(R)$, and consequently
$\Cox(\Bl_eX)$ is generated in Rees multiplicity one.
\end{proof}

We conclude with a computational application of \Cref{thm:main-fwps} to
canonical Fano tetrahedra.  These provide a finite class of fake weighted
projective threefolds in which torsion in the class group occurs naturally.
We first use Demazure-root automorphisms to reduce a large part of the
classification to the complexity-one case, and then apply
\Cref{thm:main-fwps} to the remaining cases.

\begin{proposition}
\label{prop:canonical-fano-tetrahedra}
According to the Graded Ring Database, there are $225$ canonical Fano
tetrahedra in dimension three.  For $131$ of them, an automorphism moves
the torus identity $e$ to a point on a torus-invariant curve; hence the blow-up has
complexity at most one and its Cox ring is generated in Rees multiplicity
one.  Among the remaining $94$ cases, exactly $27$ have
complete-intersection toric-point ideal, and therefore also have Cox ring
generated in Rees multiplicity one.  These $27$ cases are listed below.

\begin{center}
\scriptsize
\setlength{\tabcolsep}{2.5pt}
\renewcommand{\arraystretch}{1.7}

\begin{tabular}{
c c c |@{\qquad}
c c c |@{\qquad}
c c c
}
\toprule
ID & $\Cl(X)$ & grading
&
ID & $\Cl(X)$ & grading
&
ID & $\Cl(X)$ & grading
\\
\midrule

547306 &
$\Z\oplus\Z/2$ &
$\left[\begin{smallmatrix}
2&3&5&2\\
\bar1&\bar0&\bar1&\bar0
\end{smallmatrix}\right]$
&
547307 &
$\Z$ &
$\left[\begin{smallmatrix}6&5&9&4\end{smallmatrix}\right]$
&
547318 &
$\Z$ &
$\left[\begin{smallmatrix}6&7&11&4\end{smallmatrix}\right]$
\\

547352 &
$\Z\oplus\Z/4$ &
$\left[\begin{smallmatrix}
1&2&1&1\\
\bar2&\bar3&\bar1&\bar0
\end{smallmatrix}\right]$
&
547389 &
$\Z\oplus\Z/2\oplus\Z/4$ &
$\left[\begin{smallmatrix}
1&4&2&1\\
\bar1&\bar1&\bar0&\bar0\\
\bar3&\bar0&\bar1&\bar0
\end{smallmatrix}\right]$
&
547398 &
$\Z\oplus\Z/2$ &
$\left[\begin{smallmatrix}
7&14&3&4\\
\bar1&\bar1&\bar0&\bar0
\end{smallmatrix}\right]$
\\

547402 &
$\Z\oplus(\Z/2)^2$ &
$\left[\begin{smallmatrix}
2&6&3&1\\
\bar1&\bar1&\bar0&\bar0\\
\bar1&\bar0&\bar1&\bar0
\end{smallmatrix}\right]$
&
547412 &
$\Z\oplus\Z/2$ &
$\left[\begin{smallmatrix}
5&3&12&4\\
\bar1&\bar0&\bar1&\bar0
\end{smallmatrix}\right]$
&
547413 &
$\Z\oplus\Z/6$ &
$\left[\begin{smallmatrix}
1&1&4&2\\
\bar3&\bar5&\bar1&\bar5
\end{smallmatrix}\right]$
\\

547420 &
$\Z\oplus\Z/2$ &
$\left[\begin{smallmatrix}
3&2&10&5\\
\bar1&\bar0&\bar1&\bar0
\end{smallmatrix}\right]$
&
547423 &
$\Z\oplus\Z/6$ &
$\left[\begin{smallmatrix}
1&3&1&1\\
\bar1&\bar3&\bar2&\bar0
\end{smallmatrix}\right]$
&
547429 &
$\Z\oplus\Z/2\oplus\Z/4$ &
$\left[\begin{smallmatrix}
1&1&1&1\\
\bar1&\bar1&\bar0&\bar0\\
\bar3&\bar0&\bar1&\bar0
\end{smallmatrix}\right]$
\\

547430 &
$\Z\oplus\Z/4$ &
$\left[\begin{smallmatrix}
2&5&2&1\\
\bar1&\bar3&\bar2&\bar0
\end{smallmatrix}\right]$
&
547433 &
$\Z\oplus(\Z/2)^2$ &
$\left[\begin{smallmatrix}
1&1&4&2\\
\bar1&\bar0&\bar0&\bar1\\
\bar1&\bar0&\bar1&\bar0
\end{smallmatrix}\right]$
&
547436 &
$\Z\oplus\Z/3$ &
$\left[\begin{smallmatrix}
2&2&7&3\\
\bar0&\bar2&\bar1&\bar2
\end{smallmatrix}\right]$
\\

547437 &
$\Z\oplus\Z/2$ &
$\left[\begin{smallmatrix}
9&3&2&4\\
\bar0&\bar1&\bar0&\bar1
\end{smallmatrix}\right]$
&
547452 &
$\Z$ &
$\left[\begin{smallmatrix}6&15&5&4\end{smallmatrix}\right]$
&
547460 &
$\Z\oplus\Z/2$ &
$\left[\begin{smallmatrix}
2&3&4&3\\
\bar1&\bar1&\bar0&\bar0
\end{smallmatrix}\right]$
\\

547464 &
$\Z\oplus\Z/4$ &
$\left[\begin{smallmatrix}
2&1&1&4\\
\bar3&\bar3&\bar1&\bar3
\end{smallmatrix}\right]$
&
547465 &
$\Z\oplus\Z/4$ &
$\left[\begin{smallmatrix}
2&4&1&1\\
\bar3&\bar3&\bar2&\bar0
\end{smallmatrix}\right]$
&
547470 &
$\Z\oplus(\Z/2)^2$ &
$\left[\begin{smallmatrix}
1&3&1&1\\
\bar1&\bar1&\bar0&\bar0\\
\bar1&\bar0&\bar1&\bar0
\end{smallmatrix}\right]$
\\

547473 &
$\Z\oplus\Z/4$ &
$\left[\begin{smallmatrix}
1&3&2&2\\
\bar3&\bar0&\bar1&\bar0
\end{smallmatrix}\right]$
&
547478 &
$\Z$ &
$\left[\begin{smallmatrix}17&6&4&7\end{smallmatrix}\right]$
&
547483 &
$\Z\oplus\Z/2$ &
$\left[\begin{smallmatrix}
3&2&7&2\\
\bar0&\bar0&\bar1&\bar1
\end{smallmatrix}\right]$
\\

547495 &
$\Z\oplus\Z/4$ &
$\left[\begin{smallmatrix}
1&1&3&1\\
\bar3&\bar2&\bar1&\bar0
\end{smallmatrix}\right]$
&
547519 &
$\Z\oplus\Z/4$ &
$\left[\begin{smallmatrix}
1&1&1&1\\
\bar3&\bar2&\bar2&\bar1
\end{smallmatrix}\right]$
&
547524 &
$\Z\oplus\Z/4$ &
$\left[\begin{smallmatrix}
1&1&1&3\\
\bar3&\bar2&\bar0&\bar3
\end{smallmatrix}\right]$
\\

\bottomrule
\end{tabular}
\end{center}

\end{proposition}

\begin{proof}
Let $R=\K[x_1,x_2,x_3,x_4]$ be the Cox ring of $X$, and write
$w_i=\deg(x_i)\in\Cl(X)$.  Suppose, after reordering the variables,
that both $w_3$ and $w_4$ belong to the semigroup generated by
$w_1,w_2$.  Choose monomials $M_3,M_4\in\K[x_1,x_2]$ with
$\deg(M_3)=w_3$ and $\deg(M_4)=w_4$.  The corresponding Demazure-root
automorphisms act by
$x_3\mapsto x_3+\lambda M_3$ and
$x_4\mapsto x_4+\mu M_4$, respectively, while fixing $x_1,x_2$.
Taking suitable $\lambda$ and $\mu$ sends the Cox representative
$(1,1,1,1)$ of $e$ to $(1,1,0,0)$.

If $v_1,\ldots,v_4$ are the rays of the fan of $X$, the latter point
maps to the torus orbit corresponding to the cone
$\langle v_3,v_4\rangle$.  This orbit has dimension one, so its
connected stabilizer is a two-dimensional subtorus.  It acts on the
blow-up at this point, and hence the blow-up has complexity at most
one.  The complexity-one result therefore gives generation in Rees
multiplicity one.  A direct computation of the $225$ canonical Fano
tetrahedra shows that this applies to $131$ cases.

For each of the remaining $94$ tetrahedra we compute the lattice ideal
$I_X$.  The corresponding lattice has rank three, so
$\htop(I_X)=3$.  Exactly $27$ of these ideals have three minimal generators, and hence
are complete intersections.  By \Cref{thm:main-fwps}, their blow-ups
have Cox ring generated in Rees multiplicity one.  The grading matrices
in the table are obtained from the Smith normal form of the ray matrix.
The Magma code used for these computations is available at
\begin{center}
\href{https://github.com/alaface/multiplicity-one/blob/main/README.md}
{\texttt{github.com/alaface/multiplicity-one}}.
\end{center}
\end{proof}

\section{Projective toric surfaces}
\label{sec:surfaces}

Let $X=X_\Sigma$ be a projective toric surface, and list the primitive
ray generators $v_1,\ldots,v_r\in N\cong\Z^2$ in cyclic order.  Its
orbit lattice
\[
 L_X=\operatorname{im}\bigl(M\longrightarrow\Z^r,\quad
 m\longmapsto
 (\langle m,v_1\rangle,\ldots,\langle m,v_r\rangle)\bigr)
\]
has rank two.  Hence the toric-point ideal $I_X=I_{L_X}$ is a
codimension-two lattice ideal; since $\K$ has characteristic zero, it
is radical.
We shall use two results about codimension-two lattice ideals.

\begin{theorem}
\label{thm:ha-morales-huneke}
Let $J\subseteq\K[y_1,\ldots,y_n]$ be a codimension-two radical lattice
ideal, and let $\ell(J)$ denote its analytic spread at the homogeneous
maximal ideal.  Then
\[
 \ell(J)=
 \begin{cases}
 2,&\text{if $J$ is a complete intersection},\\
 3,&\text{otherwise}.
 \end{cases}
\]
Moreover, if $\K$ is infinite, the Rees algebra $\cR(J)$ and the
associated graded ring $\gr_J(\K[y_1,\ldots,y_n])$ are
Cohen--Macaulay.
\end{theorem}

\begin{proof}
Hà--Morales prove the analytic-spread bound and the
Cohen--Macaulayness of the Rees algebra in
\cite[Theorems~5.12 and~5.19]{HaMorales}.  If $J$ is a complete intersection of height two, then
$\gr_J(S)\cong(S/J)[T_1,T_2]$, so its fiber cone is a polynomial
ring in two variables and $\ell(J)=2$.
For a radical codimension-two lattice ideal which is not a complete
intersection, Hà--Morales prove that $\ell(J)=3$.

Finally, since both the polynomial ring and $\cR(J)$ are
Cohen--Macaulay and $J$ has positive height,
\cite[Proposition~1.1]{HunekeAssociatedGraded} implies that
$\gr_J(\K[y_1,\ldots,y_n])$ is Cohen--Macaulay.
\end{proof}

The preceding theorem makes the boundary test of
\Cref{thm:main-criterion} particularly simple for surfaces.

\begin{proposition}
\label{prop:surface-triples}
Let $A\subseteq[r]$ and suppose that the projected lattice
$L_A=\pi_A(L_X)$ is positive.  Then $\ell(I_{L_A})\leq3$, and the
inequality
\[
        \ell(I_{L_A})\leq |A|-1
\]
is automatic unless $|A|=3$.  For a three-element set $A$, this
inequality holds if and only if $I_{L_A}$ is a complete intersection.
Consequently, the numerical condition in
\Cref{thm:main-criterion} is equivalent to requiring that
$I_{L_A}$ be a complete intersection for every three-element subset
$A\subseteq[r]$ for which $L_A$ is positive.
\end{proposition}

\begin{proof}
Since $\rk(L_X)=2$, every projected lattice $L_A$ has rank at most
two.  If $\rk(L_A)=0$, then $I_{L_A}=0$ and its analytic spread is
zero.  If $\rk(L_A)=1$, then $L_A$ is cyclic, so $I_{L_A}$ is
principal and $\ell(I_{L_A})=1$.

A positive rank-two lattice cannot occur in $\Z^A$ when $|A|\leq2$:
for $|A|=2$, a rank-two sublattice has finite index in $\Z^2$ and
therefore contains a nonzero vector of $\N^2$.  Thus the only
remaining case is $\rk(L_A)=2$ and $|A|\geq3$.  By
\Cref{thm:ha-morales-huneke}, its analytic spread is either $2$ or
$3$, according as $I_{L_A}$ is or is not a complete intersection.

If $|A|\geq4$, this gives
$\ell(I_{L_A})\leq3\leq|A|-1$, so the required inequality is automatic.
For $|A|=3$, it becomes $\ell(I_{L_A})\leq2$, which by the same theorem
is equivalent to $I_{L_A}$ being a complete intersection.
\end{proof}

It remains to relate a positive three-coordinate projection to the
corresponding three-ray toric surface.

\begin{lemma}
\label{lem:triple-dehomogenization}
Let $A\subseteq[r]$ have cardinality three, suppose that $L_A$ is
positive of rank two.  Then:
\begin{enumerate}[label=\textup{(\roman*)}]
\item the projection $L_X\to L_A$ is an isomorphism;
\item setting $x_j=1$ for $j\in {A^c}$ maps $I_X$ onto $I_{L_A}$;
\item $I_X\subseteq P_A^2$ if and only if
      $I_{L_A}\subseteq\mathfrak m_A^2$;
\item the rays indexed by $A$ form a complete three-ray fan.
\end{enumerate}
\end{lemma}

\begin{proof}
The kernel of $L_X\to L_A$ is $L_X\cap\Z^{A^c}$.  Since both $L_X$ and
$L_A$ have rank two, this kernel has rank zero and is therefore
trivial.  This proves \textup{(i)}, and \textup{(ii)} follows from the
binomial description of lattice ideals.

For a lattice binomial, membership in $P_A^2$ means that each of its
two monomials has total $A$-degree at least two.  Setting the
${A^c}$-variables equal to one does not alter these degrees, and the
projection $L_X\to L_A$ is an isomorphism.  This proves
\textup{(iii)}.

Finally, positivity of $L_A$ means that there is no nonzero
$m\in M_\mathbb R$ which is nonnegative on all three selected rays.
By the Gordan--Stiemke theorem of alternatives, there are positive
rational numbers $c_i$, $i\in A$, such that
$\sum_{i\in A}c_iv_i=0$.  Thus the three rays positively span
$N_\mathbb R$ and determine a complete fan.
\end{proof}

\begin{proof}[Proof of \Cref{thm:main-surface}]
The implication \textup{(i)}$\Rightarrow$\textup{(ii)} is
\cite[Theorem~3.5]{LU-minimal}.  We prove the converse.

Assume \textup{(ii)}.  By \Cref{prop:surface-triples}, it is enough to
show that $I_{L_A}$ is a complete intersection whenever $A$ has three
elements and $L_A$ is positive.  There is nothing to prove if
$\rk(L_A)\leq1$, so suppose that $\rk(L_A)=2$.

By \Cref{lem:triple-dehomogenization}, the assumption
$I_X\not\subseteq P_A^2$ gives
$I_{L_A}\not\subseteq\mathfrak m_A^2$, and the selected rays form a
complete three-ray fan.  Hence
\cite[Proposition~3.4]{LU-minimal} implies that $I_{L_A}$ is a
complete intersection.  The numerical inequalities of
\Cref{thm:main-criterion} therefore hold for every $A$.

Since $I_X$ is a codimension-two radical lattice ideal,
\Cref{thm:ha-morales-huneke} shows that
$\gr_{I_X}(R)$ is Cohen--Macaulay, hence unmixed.
The unmixed part of \Cref{thm:main-criterion} now gives
multiplicity-one generation.  Thus
\textup{(ii)}$\Rightarrow$\textup{(i)}.

We finally compare \textup{(ii)} and \textup{(iii)}.  Suppose first
that a triple $A$ determines a complete three-ray fan.  Then $L_A$ is
positive of rank two, and the fan defines a rank-one toric surface
$X_A$ together with the toric contraction $X\to X_A$.
By \Cref{lem:triple-dehomogenization}, dehomogenizing the variables
outside $A$ identifies the toric-point ideal of $X_A$ with $I_{L_A}$
and gives
\[
 I_X\not\subseteq P_A^2
 \quad\Longleftrightarrow\quad
 I_{L_A}\not\subseteq\mathfrak m_A^2.
\]
By \cite[Proposition~3.4]{LU-minimal}, the latter condition is
equivalent to multiplicity-one generation for $\Bl_eX_A$.

If the selected rays do not form a complete fan, then $L_A$ is not
positive.  Choose a nonzero vector in $L_A\cap\N^A$ and lift it to
$u\in L_X$.  In the corresponding lattice binomial one monomial has
total $A$-degree zero, so the binomial does not belong to $P_A^2$.
Thus condition \textup{(ii)} is automatic for such triples.  The only
triples that need to be tested are therefore exactly the three-ray
contractions occurring in \textup{(iii)}, which proves
\textup{(ii)}$\Leftrightarrow$\textup{(iii)}.
\end{proof}

We now specialize to minimal toric surfaces of Picard rank two.  Such
a surface has four rays.  After a change of lattice basis, their
primitive generators can be written in cyclic order as
\[
       (1,0),\quad (a,b),\quad (-1,0),\quad (-a,-b),
        \qquad b>0,\quad\gcd(a,b)=1;
\]
see \cite[Section~4]{LU-minimal}.  The orbit lattice contains
$(b,0,-b,0)$ and $(0,b,0,-b)$, and hence
\begin{equation}
\label{eq:opposite-ray-binomials}
        x_1^b-x_3^b,\qquad x_2^b-x_4^b\in I_X.
\end{equation}

\begin{proof}[Proof of \Cref{cor:main-minimal}]
Let $A\subseteq\{1,2,3,4\}$ be a triple and let $j$ be the omitted
index.  If $j=1$ or $j=3$, the first binomial in
\eqref{eq:opposite-ray-binomials} has a monomial involving only the
omitted variable, so it does not belong to $P_A^2$.  If $j=2$ or
$j=4$, the second binomial gives the same conclusion.  Thus
$I_X\not\subseteq P_A^2$ for every triple, and
\Cref{thm:main-surface} applies.
\end{proof}

\subsection{Partial resolutions of Gorenstein toric Fano surfaces}

A lattice polygon $P\subset N_\mathbb R$ is \emph{reflexive} if the origin
is its unique interior lattice point and its polar
\[
        P^\vee
        =\{m\in M_\mathbb R:\langle m,v\rangle\geq-1
          \text{ for every }v\in P\}
\]
is again a lattice polygon.  If $Y$ is a Gorenstein toric Fano surface, the
convex hull $P$ of the primitive ray generators of its fan is reflexive.  A
toric partial resolution $X\to Y$ is obtained by subdividing the fan using
some of the rays through lattice points of $\partial P$.

Let
\[
        Q_{\mathrm{exc}}
        =\operatorname{conv}(e_1,e_2,-e_1-e_2)\subset M_\mathbb R
\]
and let $P_{\mathrm{exc}}=Q_{\mathrm{exc}}^\vee$.  The vertices of the polar
triangle are
\begin{equation}\label{eq:exceptional-rays}
        (2,-1),\qquad (-1,-1),\qquad (-1,2).
\end{equation}
We denote by $X_{\mathrm{exc}}$ the rank-one toric surface defined by the
complete fan with these three rays.

\begin{proposition}
\label{prop:gorenstein-partial-resolutions}
Let $X\to Y$ be a toric partial resolution of a Gorenstein toric Fano surface.
Then
\[
 \Cox(\Bl_eX)\text{ is generated in Rees multiplicity one}
 \quad\Longleftrightarrow\quad
 X\text{ does not dominate }X_{\mathrm{exc}}.
\]
Here domination means that the fan of $X$ refines the three-ray fan of
$X_{\mathrm{exc}}$, so that the identity on the lattice induces a toric
birational morphism $X\to X_{\mathrm{exc}}$.
\end{proposition}

\begin{proof}
Let $P$ be the reflexive polygon associated with $Y$.  By
\Cref{thm:main-surface}, it is enough to consider three rays of the fan of
$X$ that form a complete fan.  Write their primitive generators as
$v_1,v_2,v_3\in\partial P$.

Suppose that a facet $F$ of $P$ contains exactly one of these points, say
$v_1$.  Let $m\in P^\vee$ be the vertex dual to $F$.  Then
\[
        \langle m,v_1\rangle=-1,
        \qquad
        \langle m,v_2\rangle,\langle m,v_3\rangle\geq0.
\]
The lattice binomial associated with $m$ therefore has $x_1$ as one of its
monomials.  In particular, the toric-point ideal of the three-ray surface is
not contained in $(x_1,x_2,x_3)^2$, so this triple causes no obstruction.

We may consequently assume that no facet of $P$ contains exactly one of the
three selected points.  Arrange $v_1,v_2,v_3$ in cyclic order along
$\partial P$.  Consecutive selected points must lie on a common facet:
otherwise the first facet leaving one of them toward the next would contain
exactly that selected point.  Thus each of the pairs
$(v_1,v_2)$, $(v_2,v_3)$, and $(v_3,v_1)$ lies on a facet.  Since the three
vectors positively span $N_\mathbb R$, these facets are distinct.  It follows
that $P$ is a triangle and that $v_1,v_2,v_3$ are its vertices.

Set $Q=P^\vee$.  If one side of $Q$ contains an interior lattice point $m$,
then that side is dual to one of the vertices, say $v_1$, and
\[
        \langle m,v_1\rangle=-1,
        \qquad
        \langle m,v_2\rangle,\langle m,v_3\rangle\geq0.
\]
Again the associated lattice binomial has a linear term, so there is no
obstruction.

The only remaining possibility is that every side of the reflexive triangle
$Q$ has lattice length one.  Then $Q$ has three boundary lattice points and
one interior lattice point.  Pick's theorem gives normalized area three.
The three triangles obtained by joining the origin to the sides of $Q$ have
positive integral normalized areas summing to three, so each has area one.
After a lattice change of coordinates, two consecutive vertices are
$e_1,e_2$; the determinant-one conditions for the other two sides force the
third vertex to be $-e_1-e_2$.  Hence
$Q\cong Q_{\mathrm{exc}}$ and $P\cong P_{\mathrm{exc}}$.

For the exceptional fan, the toric-point ideal is
\begin{equation}\label{eq:exceptional-ideal}
 I_{\mathrm{exc}}
 =\bigl(
 x_1x_3-x_2^2,\;
 x_1x_2-x_3^2,\;
 x_1^2-x_2x_3
 \bigr).
\end{equation}
It is contained in $(x_1,x_2,x_3)^2$, so
\Cref{thm:main-surface} shows that multiplicity-one generation fails.  Thus a
partial resolution fails the triple test exactly when its fan contains a
three-ray coarsening lattice-equivalent to the exceptional fan, which is
precisely the asserted domination condition.
\end{proof}

\begin{figure}[ht]
\centering
\begin{tikzpicture}[scale=1.05]
  \begin{scope}[shift={(-3.2,0)}]
    \draw[->,gray!60] (-1.8,0)--(1.8,0);
    \draw[->,gray!60] (0,-1.8)--(0,1.8);
    \foreach \i in {-1,0,1}
      \foreach \j in {-1,0,1}
        \fill (\i,\j) circle (0.035);
    \draw[very thick] (1,0)--(0,1)--(-1,-1)--cycle;
    \fill (1,0) circle (0.07) node[above right] {$e_1$};
    \fill (0,1) circle (0.07) node[above right] {$e_2$};
    \fill (-1,-1) circle (0.07) node[below left] {$-e_1-e_2$};
    \node at (0,-2.15) {$Q_{\mathrm{exc}}$};
  \end{scope}
  \begin{scope}[shift={(3.2,0)},scale=0.62]
    \draw[->,gray!60] (-2.5,0)--(2.8,0);
    \draw[->,gray!60] (0,-2.5)--(0,2.8);
    \foreach \i in {-2,-1,0,1,2}
      \foreach \j in {-2,-1,0,1,2}
        \fill (\i,\j) circle (0.045);
    \draw[very thick] (2,-1)--(-1,-1)--(-1,2)--cycle;
    \fill (2,-1) circle (0.10) node[right] {$(2,-1)$};
    \fill (-1,-1) circle (0.10) node[below left] {$(-1,-1)$};
    \fill (-1,2) circle (0.10) node[above left] {$(-1,2)$};
    \node at (0,-3.35) {$P_{\mathrm{exc}}=Q_{\mathrm{exc}}^\vee$};
  \end{scope}
\end{tikzpicture}
\caption{The exceptional reflexive triangle and its polar.  The vertices of
the polar are the ray generators of $X_{\mathrm{exc}}$.}
\label{fig:exceptional-polar-pair}
\end{figure}
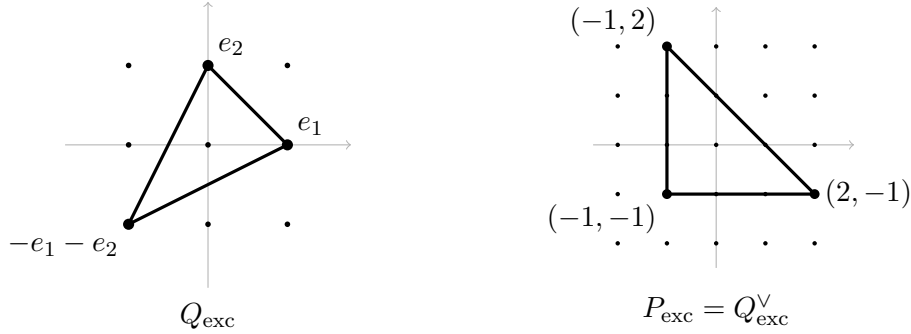

The exceptional case admits a completely explicit saturated Rees algebra.

\begin{proposition}
\label{prop:exceptional-cox-ring}
Let $R=\K[x_1,x_2,x_3]$, let $I_{\mathrm{exc}}$ be the ideal in
\eqref{eq:exceptional-ideal}, and put
\[
        g=x_1^3+x_2^3+x_3^3-3x_1x_2x_3.
\]
Then
\begin{equation}\label{eq:exceptional-saturated-powers}
 I_{\mathrm{exc}}^m:(x_1,x_2,x_3)^\infty
 =\sum_{q=0}^{\lfloor m/2\rfloor}
   g^q I_{\mathrm{exc}}^{m-2q}
 \qquad(m\geq0).
\end{equation}
Consequently,
\[
 \Cox(\Bl_eX_{\mathrm{exc}})
 =R\bigl[t,I_{\mathrm{exc}}t^{-1},g t^{-2}\bigr].
\]
The element $gt^{-2}$ is necessary, so the first additional Rees
multiplicity is exactly two.
\end{proposition}

\begin{proof}
Choose a primitive cube root of unity $\omega\in\K$.  The projective scheme
defined by $I_{\mathrm{exc}}$ consists of the three reduced points
\[
        [1:1:1],\qquad
        [1:\omega:\omega^2],\qquad
        [1:\omega^2:\omega].
\]
The three linear forms
\[
 \begin{aligned}
 u&=x_1+x_2+x_3,\\
 v&=x_1+\omega x_2+\omega^2x_3,\\
 w&=x_1+\omega^2x_2+\omega x_3
 \end{aligned}
\]
form a basis of $R_1$ and carry these points to the three coordinate points
of $\PP^2$.  Under this linear change of coordinates,
$I_{\mathrm{exc}}$ becomes
\[
        J=(uv,uw,vw),
\]
and
\[
        uvw=x_1^3+x_2^3+x_3^3-3x_1x_2x_3=g.
\]

The saturation of $J^m$ is its symbolic power
\[
 J^{(m)}=(u,v)^m\cap(u,w)^m\cap(v,w)^m.
\]
A monomial $u^av^bw^c$ lies in $J^{(m)}$ exactly when
\[
        a+b\geq m,\qquad a+c\geq m,\qquad b+c\geq m.
\]
If one exponent is zero, the monomial is divisible by one of
$(uv)^m,(uw)^m,(vw)^m$ and therefore belongs to $J^m$.  If all three
exponents are positive, division by $uvw$ reduces the three inequalities
from level $m$ to level $m-2$.  Induction on $m$ gives
\[
        J^{(m)}
        =\sum_{q=0}^{\lfloor m/2\rfloor}(uvw)^qJ^{m-2q}.
\]
Undoing the linear change of coordinates proves
\eqref{eq:exceptional-saturated-powers}.  The saturated extended Rees algebra
is therefore generated in multiplicities one and two.  Finally,
$I_{\mathrm{exc}}$ is generated by quadrics, so
$I_{\mathrm{exc}}^2$ contains no nonzero cubic.  Hence $g\notin
I_{\mathrm{exc}}^2$, and $gt^{-2}$ cannot be omitted.
\end{proof}

\bibliographystyle{amsplain}
\bibliography{references}

\end{document}